\documentclass{amsart}

\usepackage{delimset,amssymb,enumitem,mathtools}
\setlist{itemsep=4pt,topsep=2pt,leftmargin=17pt,listparindent=11pt}
\usepackage{booktabs,longtable,array}
\usepackage{graphicx}
\usepackage{xcolor}
\usepackage[colorlinks=true,allcolors=black]{hyperref}
\usepackage{aliascnt}

\usepackage[capitalize]{cleveref}
\usepackage{subcaption}
\usepackage{tikz}
\usetikzlibrary{positioning,calc,fit,arrows.meta}
\usepackage[numbers,sort&compress,nonamebreak,merge,elide,longnamesfirst]{natbib}
\usepackage{etoolbox}
\usepackage{hyphenat}
\usepackage{microtype}

\makeatletter
\patchcmd{\NAT@citexnum}
{\@citea \NAT@test{\@ne}\NAT@spacechar\NAT@mbox{\NAT@super@kern\NAT@@open}}
{\@citea \nohyphens{\NAT@test{\@ne}}\nobreakspace{}\NAT@mbox{\NAT@super@kern\NAT@@open}}
{}
{\PackageWarning{nonbreaking-citet}{Failed to patch \string\citet}}
\makeatother

\newtheorem{theorem}{Theorem}[section]

\newtheorem{proposition}[theorem]{Proposition}

\theoremstyle{definition}
\newtheorem{definition}[theorem]{Definition}

\theoremstyle{remark}

\numberwithin{equation}{section}
\numberwithin{table}{section}

\crefname{conjecture}{Conjecture}{Conjectures}
\crefname{definition}{Definition}{Definitions}
\crefname{corollary}{Corollary}{Corollaries}
\crefname{remark}{Remark}{Remarks}
\crefname{question}{Question}{Questions}

\allowdisplaybreaks

\tikzset{
  edge/.style={semithick},
  vertex/.style={circle,fill=black,inner sep=1.7pt},
  hnode/.style={circle,draw,inner sep=2.3pt,font=\small},
  boxnode/.style={draw,rounded corners,inner sep=2.5pt,font=\small}
}

\newcommand{\N}{\mathbb N}

\title[Three infinite separating families]{Three Infinite Families Separating Schur Positivity, the Strongly Nice Property, and the Nice Property}

\author[K. Zhang]{K. Zhang}
\address[K. Zhang]{School of Mathematics and Statistics, Beijing Institute of Technology, Beijing 102400, P.\ R.\ China.}
\email{kai@bit.edu.cn}


\begin{document}

\begin{abstract}
For a graph $G$, Schur positivity of $X_G$ implies that $G$ is strongly nice, and every strongly nice graph is nice. We construct three infinite families separating these properties. We first give a connected family $F_t$, $t\ge6$, that is strongly nice but not Schur positive. We then prove that homogeneous strongly nice symmetric functions with nonnegative monomial coefficients are closed under multiplication, and hence that strongly nice graphs are closed under disjoint union. As an application, for $H=K_{3,3}-e$, the graphs
\[
M_t=H\sqcup K_t,\qquad t\ge3,
\]
form a disconnected family that is strongly nice but not Schur positive. Finally, we define
\[
N_r=K_r\vee(K_2\sqcup2K_1),\qquad r\ge2,
\]
and prove that every $N_r$ is connected and nice but not strongly nice. We also introduce the level-$k$ nice property and show that the level depth of $N_r$ is $4r!$.
\end{abstract}

\maketitle

\section{Introduction}\label{sec:intro}
For a finite simple graph $G=(V,E)$, Stanley~\cite{Stanley1995}
defined the chromatic symmetric function
\[
 X_G=\sum_{\kappa}\prod_{v\in V}x_{\kappa(v)},
\]
where the sum ranges over all proper colorings $\kappa:V\to\N$.
The ring of symmetric functions has several classical bases, including
the elementary, monomial, and Schur bases, denoted by
$\{e_\lambda\}$, $\{m_\lambda\}$, and $\{s_\lambda\}$, respectively.
Standard references include Macdonald~\cite{Macdonald1995} and
Stanley~\cite{StanleyEC2}.
A graph is called $e$-positive (respectively, $s$-positive) if all
coefficients of its chromatic symmetric function in the elementary
basis (respectively, Schur basis) are nonnegative. Every $e$-positive
symmetric function is Schur positive; see Stanley~\cite{StanleyEC2}.

Write
\[
X_G=\sum_{\lambda\vdash |V(G)|} a_\lambda(G)m_\lambda .
\]
Stanley~\cite{Stanley1998} defined a graph $G$ to be \emph{nice} if,
for any partitions $\lambda,\mu\vdash |V(G)|$ with
$\lambda\ge\mu$ in dominance order, the positivity of
$a_\lambda(G)$ implies the positivity of $a_\mu(G)$.
Li, Li, Yang, and Zhang~\cite{LiEtAlStrong} defined $G$ to be
\emph{strongly nice} if $a_\lambda(G)\le a_\mu(G)$ whenever
$\lambda,\mu\vdash |V(G)|$ satisfy $\lambda\ge\mu$ in dominance order.
Li, Li, Yang, and Zhang~\cite{LiEtAlStrong} proved that every
$s$-positive graph is strongly nice and every strongly nice graph is
nice. Their proof of the former statement uses the monotonicity of
Kostka numbers established by White~\cite{White1980}.
These known relations are illustrated in Figure~\ref{fig:classes}.

\begin{figure}[htbp]
\centering
\begin{tikzpicture}[scale=0.8]
\draw[line width=1pt] (-0.7,0) ellipse [x radius=5.7, y radius=1.95];
\draw[line width=1pt] (-1.9,0) ellipse [x radius=4.5, y radius=1.45];
\draw[line width=1pt] (-3.9,0) ellipse [x radius=2.5, y radius=1.0];
\draw[line width=1pt] (-4.92,0) ellipse [x radius=1.45, y radius=0.68];
\node at (-4.92,0) {\large $e$-pos};
\node at (-2.5,0) {\large $s$-pos};
\node at (0.3,0) {\large strongly nice};
\node at (3.5,0) {\large nice};
\end{tikzpicture}
\caption{The inclusion relations among the four graph classes.}
\label{fig:classes}
\end{figure}
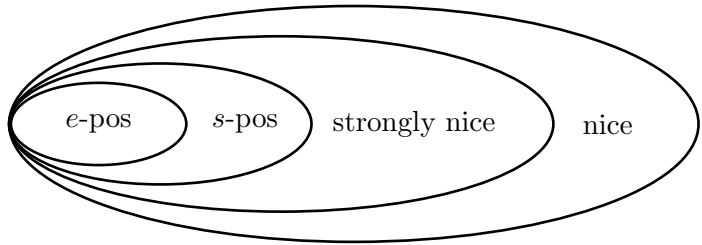

Several families of Schur positive chromatic symmetric functions are known. Gasharov~\cite{Gasharov1996} proved that the incomparability graph of every $(3+1)$-free poset is Schur positive. Hikita~\cite{Hikita2025} proved the Stanley--Stembridge conjecture, which gives the stronger $e$-positivity for this class. Shelburne and van Willigenburg~\cite{ShelburneVanW2025} proved Schur positivity for generalized net graphs and later characterized the Schur positive complete multipartite graphs~\cite{ShelburneVanW2026}.

Wang and Wang~\cite{WangWang2020} gave a combinatorial formula for every Schur coefficient of a chromatic symmetric function. Let $\mathcal T_\lambda$ denote the set of special rim-hook tabloids of shape $\lambda$. For $T\in\mathcal T_\lambda$, let $\lambda(T)$ be the partition formed by the ribbon sizes of $T$, and let $\operatorname{ht}(T)$ be the sum of the ribbon heights.

\begin{theorem}[Wang--Wang \cite{WangWang2020}]\label{thm:intro-WW}
For every graph $G$ and every partition $\lambda$,
\[
 [s_\lambda]X_G
 =\sum_{T\in\mathcal T_\lambda}
 (-1)^{\operatorname{ht}(T)}\,a_{\lambda(T)}(G).
\]
\end{theorem}

Wang and Wang~\cite{WangWang2020} used this formula to study the
Schur positivity of several graph families. In particular, they
completely characterized the Schur positive fan and complete
tripartite graphs, and showed that certain families of squid and
pineapple graphs are not $s$-positive. They also conjectured that the squid graph $\operatorname{Sq}(2n-1;1^n)$ is not Schur positive for every $n\ge3$~\cite{WangWang2020}. Li, Li, Yang, and Zhang~\cite{LiEtAlStrong} proved this conjecture by showing that these squid graphs are not strongly nice; they also proved that the same graphs are nice. Their paper also contains a six-vertex graph that is strongly nice but not Schur positive~\cite{LiEtAlStrong}; this graph is isomorphic to $K_{3,3}-e$.

Stanley~\cite{Stanley1998}, attributing the conjecture to Gasharov~\cite{Gasharov1999}, stated that every claw-free graph is Schur positive. Prajapati~\cite{Prajapati2026} and Matherne and Morales~\cite{MatherneMorales2026} independently disproved the conjecture by finding the same two $12$-vertex line-graph counterexamples. Wang, Zhang, and Zhao~\cite{WangZhangZhao2026} subsequently constructed two infinite families of counterexamples: every graph in the first family is a line graph, while no graph in the second family is a line graph. They also determined minimum counterexamples under two lexicographic orders and proved a Schur-coefficient transfer formula for graphs obtained by adjoining a clique through a single bridge~\cite{WangZhangZhao2026}.

For distributive lattices, Lonc and Elzobi~\cite{LoncElzobi1999}
characterized chain partitions of products of two chains.
Li, Qiu, Yang, and Zhang~\cite{LQYZ2025} explained that this
characterization implies that products of two chains are nice.
Li, Qiu, Yang, and Zhang~\cite{LQYZ2025} also constructed a family of distributive lattices that are not
nice, and hence not Schur positive, as well as an infinite family of
products of two chains that are nice but not Schur positive.
Wang and Zhang~\cite{WangZhang2026} later determined the exact
Schur-positivity thresholds for $\mathbf m\times\mathbf 2$ and
$\mathbf m\times\mathbf 3$, and proved that
$\mathbf m\times\mathbf 3$ is not strongly nice for $m\ge 44$.
The proofs that the above products of two chains are not $s$-positive use the
Wang--Wang Schur-coefficient formula~\cite{WangWang2020}.

The known examples show that the converses among Schur positivity, the strongly nice property, and the nice property fail, but they leave several structural questions. Li, Li, Yang, and Zhang~\cite{LiEtAlStrong} exhibited the six-vertex graph $K_{3,3}-e$ that is strongly nice but not Schur positive. It is natural to ask whether this phenomenon persists in an infinite family of connected graphs. A second question concerns disjoint unions. The strongly nice property is defined by monomial-coefficient inequalities, whereas Schur coefficients multiply through Littlewood--Richardson convolution, so the two properties need not behave in the same way under disjoint union. Finally, the nice property records only which monomial coefficients are nonzero, whereas the strongly nice property compares their actual values along dominance order. This suggests introducing intermediate conditions that retain a finite amount of coefficient information.

These questions motivate the three families considered below. The first is a connected family of strongly nice graphs that are not Schur positive. For the second, we prove a multiplicative closure theorem for homogeneous strongly nice symmetric functions with nonnegative monomial coefficients, and hence closure under disjoint union for strongly nice graphs; we also determine exactly when $H\sqcup K_t$ is Schur positive. The third is a connected family of nice graphs that are not strongly nice and leads to a quantitative refinement between the nice and strongly nice properties.

For the first family, let
\[
 H=K_{3,3}-e.
\]
Fix an endpoint $a$ of the deleted edge, take a disjoint complete graph
\[
 K_t=\{w,c_2,\ldots,c_t\},
\]
and add only the single edge $aw$. Denote the resulting connected graph by $F_t$.

\begin{theorem}\label{thm:intro-F}
For every $t\ge6$, the graph $F_t$ is connected and strongly nice, but it is not Schur positive. More precisely,
\[
 [s_{(3,3,3,1^{t-3})}]X_{F_t}
 =-4(t-1)(t-1)!<0.
\]
\end{theorem}

Thus the separation between the strongly nice property and Schur positivity is not confined to a single small example: it persists in an infinite family of connected graphs. To prove that $F_t$ is strongly nice, we classify the stable partitions of the six-vertex graph $H$ and lift them through the clique by a restricted injection rule. This gives exactly $23$ nonzero monomial types and formulas for all corresponding coefficients. Their dominance poset has $33$ covering relations. We verify the required coefficient inequality on these $33$ covers and use transitivity. The displayed negative Schur coefficient follows from the bridge Schur-coefficient transfer formula of Wang, Zhang, and Zhao~\cite{WangZhangZhao2026}.

For the second construction, we first prove that the product of two homogeneous strongly nice symmetric functions with nonnegative monomial coefficients is strongly nice. Hence strongly nice graphs are closed under disjoint union. We then define
\[
 M_t=H\sqcup K_t,\qquad t\ge1.
\]

\begin{theorem}\label{thm:intro-M}
For every $t\ge1$, the graph $M_t$ is disconnected and strongly nice. Moreover, $M_t$ is Schur positive if and only if $t=1$ or $t=2$. For every $t\ge3$,
\[
 [s_{(3,3,3,1^{t-3})}]X_{M_t}=-4t!<0.
\]
\end{theorem}

Thus a non-Schur-positive component need not force the whole disjoint union to be non-Schur-positive: both $H\sqcup K_1$ and $H\sqcup K_2$ are Schur positive. The proof uses a short Pieri criterion for $G\sqcup K_t$; for $t\ge3$, the negative coefficient $[s_{222}]X_H=-4$ is isolated by the shape $(3,3,3,1^{t-3})$.

For the third family, recall that the join $G\vee H$ of two vertex-disjoint graphs is obtained from $G\sqcup H$ by adding every edge between the two vertex sets. Define
\[
 N_r=K_r\vee(K_2\sqcup2K_1),\qquad r\ge2.
\]

\begin{theorem}\label{thm:intro-N}
For every $r\ge2$, the graph $N_r$ is connected and nice but not strongly nice.
\end{theorem}

The family $N_r$ allows the difference between the nice and strongly nice properties to be measured explicitly. For $k\ge1$, we call $G$ \emph{level-$k$ nice} if, for every pair of partitions $\lambda,\mu$ with $\lambda\ge\mu$ in dominance order and every integer $1\le j\le k$, the inequality $a_\lambda(G)\ge j$ implies $a_\mu(G)\ge j$. The level-$1$ nice property is the nice property, while a graph is strongly nice if and only if it is level-$k$ nice for every $k$. Define
\[
 d(G):=\sup\bigl(\{k\in\N:G\text{ is level-}k\text{ nice}\}\cup\{0\}\bigr),
\]
with $d(G)=\infty$ when the condition holds for every $k$. We prove
\[
 d(N_r)=4r!.
\]
In particular, the level depth of $N_r$ tends to infinity although no $N_r$ is strongly nice. Hence, for every fixed $k$, there are connected graphs that are level-$k$ nice but not strongly nice; no fixed finite level gives an equivalent formulation of the strongly nice property.

The paper is organized as follows. \Cref{sec:prelim} gives the definitions and symmetric-function facts used below, including stable partitions, semi-ordered stable partitions, the vertical Pieri rule, and multiplicativity of chromatic symmetric functions on disjoint unions. \Cref{sec:strong-not-s} constructs the connected family $F_t$ and proves that it is strongly nice but not Schur positive. \Cref{sec:disconnected-strong-not-s} proves the multiplicative closure theorem and studies Schur positivity of $H\sqcup K_t$, yielding the disconnected separating family for $t\ge3$. \Cref{sec:nice-not-strong} studies the connected family $N_r$ and introduces the level-$k$ nice property. The appendix records the $33$ normalized coefficient differences used to verify that $F_t$ is strongly nice.

\section{Preliminaries}\label{sec:prelim}

\subsection{Partitions and dominance order}

A \emph{partition} of a positive integer $n$ is a weakly decreasing sequence
\[
 \lambda=(\lambda_1,\ldots,\lambda_\ell),\qquad
 \lambda_1\ge\cdots\ge\lambda_\ell>0,
 \qquad \sum_{i=1}^\ell\lambda_i=n.
\]
We write $\lambda\vdash n$, $|\lambda|=n$, and $\ell(\lambda)=\ell$.  We set $\lambda_i=0$ when $i>\ell(\lambda)$.  Exponents indicate repeated parts; for instance, $(3,1^4)=(3,1,1,1,1)$.

For partitions $\lambda,\mu\vdash n$, we say that $\lambda$ \emph{dominates} $\mu$, written $\lambda\ge\mu$, if
\[
 \sum_{i=1}^k\lambda_i\ge\sum_{i=1}^k\mu_i
 \qquad\text{for every }k\ge1.
\]
A covering relation is written $\lambda\gtrdot\mu$: this means $\lambda>\mu$ and there is no partition $\nu$ with $\lambda>\nu>\mu$. Every dominance relation between integer partitions can be obtained by finitely many unit balancing transfers, each replacing two parts $p\ge q+2$ by $p-1$ and $q+1$ and then reordering; see \cite{MarshallOlkinArnold2011}.

\subsection{Monomial and Schur bases, Kostka numbers, and the Pieri rule}

Let $\Lambda$ be the ring of symmetric functions over $\mathbb Q$ in variables $x_1,x_2,\ldots$. We use the monomial and Schur bases, denoted by
\[
 \{m_\lambda\},\qquad \{s_\lambda\}.
\]
For $\lambda=(\lambda_1,\ldots,\lambda_\ell)$, the monomial symmetric function is
\[
 m_\lambda=\sum_\alpha x^\alpha,
\]
where $\alpha$ ranges over the distinct permutations of
$(\lambda_1,\ldots,\lambda_\ell,0,0,\ldots)$ and
$x^\alpha=x_1^{\alpha_1}x_2^{\alpha_2}\cdots$.

The Schur and monomial bases are related by
\[
 s_\nu=\sum_{\mu}K_{\nu\mu}m_\mu,
\]
where $K_{\nu\mu}$ is the Kostka number. We write $[b_\lambda]f$ for the coefficient of $b_\lambda$ in the expansion of $f$ in a basis $\{b_\lambda\}$.

We will also use the vertical Pieri rule. For a partition $\lambda$ and an integer $r\ge1$,
\[
 s_\lambda e_r=\sum_\mu s_\mu,
\]
where the sum ranges over all partitions $\mu$ for which $\mu/\lambda$ is a vertical strip of size $r$, that is, a skew diagram with at most one cell in each row. Standard references for these facts include Macdonald~\cite{Macdonald1995} and Stanley~\cite{StanleyEC2}.

\subsection{Chromatic symmetric functions and stable partitions}

\begin{definition}\label{def:stable-partition}
Let $G=(V,E)$ be a finite simple graph. A subset $B\subseteq V$ is \emph{stable} if no two vertices of $B$ are adjacent. A \emph{stable partition} of $G$ is a set partition
\[
 \mathcal B=\{B_1,\ldots,B_k\}
\]
of $V$ such that every block $B_i$ is stable. Its \emph{type} is the partition obtained by arranging $|B_1|,\ldots,|B_k|$ in weakly decreasing order.

A \emph{semi-ordered stable partition} is a stable partition in which the blocks of each fixed size are linearly ordered.
\end{definition}

If a stable partition has type $\lambda$ and $m_j(\lambda)$ denotes the multiplicity of the part $j$ in $\lambda$, then it determines
\begin{equation}\label{eq:semi-factor}
 \prod_{j\ge1}m_j(\lambda)!
\end{equation}
semi-ordered stable partitions.

Stanley gave the following combinatorial monomial expansion of $X_G$ \cite{Stanley1995}.

\begin{proposition}[Stanley]\label{prop:stanley-m}
Let $\widetilde a_\lambda(G)$ be the number of semi-ordered stable partitions of $G$ of type $\lambda$.  Then
\[
 X_G=\sum_{\lambda\vdash |V|}\widetilde a_\lambda(G)m_\lambda.
\]
Consequently,
\[
 [m_\lambda]X_G=\widetilde a_\lambda(G)\in\mathbb Z_{\ge0}.
\]
\end{proposition}

We abbreviate
\[
 a_\lambda(G):=[m_\lambda]X_G.
\]

For vertex-disjoint graphs $G_1$ and $G_2$, the definition of the chromatic symmetric function gives
\[
 X_{G_1\sqcup G_2}=X_{G_1}X_{G_2},
\]
since a proper coloring of the disjoint union is obtained by choosing proper colorings of its two components independently. This multiplicativity will be used in \cref{sec:disconnected-strong-not-s}.

\subsection{Nice and strongly nice}

\begin{definition}[Stanley~\cite{Stanley1998}]\label{def:nice}
A graph $G$ is \emph{nice} if, for all partitions
$\lambda,\mu\vdash |V(G)|$ with $\lambda\ge\mu$ in dominance order,
the positivity of $a_\lambda(G)$ implies the positivity of
$a_\mu(G)$.
Equivalently, every stable-partition type occurring above $\mu$ in
dominance order forces the type $\mu$ to occur.
\end{definition}

\begin{definition}[Li--Li--Yang--Zhang~\cite{LiEtAlStrong}]
\label{def:strong}
A homogeneous symmetric function
$f=\sum_\lambda a_\lambda m_\lambda$ is \emph{strongly nice} if
$a_\lambda\le a_\mu$ for all partitions $\lambda,\mu$ satisfying
$\lambda\ge\mu$ in dominance order.
A graph $G$ is strongly nice if $X_G$ is strongly nice.
\end{definition}

The following two results are used throughout the paper.

\begin{theorem}[Stanley~\cite{Stanley1998}]\label{thm:stanley-s-nice}
Every Schur positive graph is nice.
\end{theorem}

\begin{theorem}[Li--Li--Yang--Zhang~\cite{LiEtAlStrong}]\label{thm:li-s-strong}
Every Schur positive symmetric function is strongly nice. Consequently, every Schur positive graph is strongly nice, and every strongly nice graph is nice.
\end{theorem}

The proof of \cref{thm:li-s-strong} uses White's~\cite{White1980}
monotonicity theorem for Kostka numbers.

\section{Connected strongly nice but not Schur positive graphs}\label{sec:strong-not-s}

\subsection{The seed graph and the connected bridge family}\label{subsec:Hseed}

Let $H=K_{3,3}-e$, with bipartition
$\{a,x,y\}\sqcup\{b,u,v\}$ and deleted edge $ab$.
Li, Li, Yang, and Zhang~\cite{LiEtAlStrong} showed that $H$ is strongly nice and computed
\begin{equation}\label{eq:H-m}
X_H=720m_{1^6}+168m_{21111}+44m_{2211}+6m_{222}
    +12m_{3111}+6m_{321}+2m_{33},
\end{equation}
and
\begin{equation}\label{eq:H-s}
X_H=152s_{1^6}+52s_{21111}+26s_{2211}-4s_{222}
    +2s_{3111}+4s_{321}+2s_{33}.
\end{equation}
In particular, $[s_{222}]X_H=-4$, so $H$ is not Schur positive.

Let $K_t=\{w,c_2,\ldots,c_t\}$ be disjoint from $H$ and add the single
edge $aw$.  Denote the resulting graph by $F_t$; see \cref{fig:Ft}.

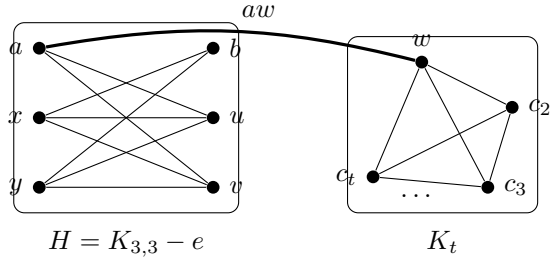
\begin{figure}[htbp]
\centering
\begin{tikzpicture}[scale=.92]
\node[vertex,label=left:$a$] (a) at (0,2) {};
\node[vertex,label=left:$x$] (x) at (0,1) {};
\node[vertex,label=left:$y$] (y) at (0,0) {};
\node[vertex,label=right:$b$] (b) at (2.5,2) {};
\node[vertex,label=right:$u$] (u) at (2.5,1) {};
\node[vertex,label=right:$v$] (v) at (2.5,0) {};
\foreach \R in {u,v}{\draw (a)--(\R);}
\foreach \L in {x,y}{\foreach \R in {b,u,v}{\draw (\L)--(\R);}}
\node[draw,rounded corners,fit=(a)(x)(y)(b)(u)(v),inner sep=7pt] (Hbox) {};
\node[below=3pt of Hbox] {$H=K_{3,3}-e$};
\node[vertex,label=above:$w$] (w) at (5.5,1.8) {};
\node[vertex,label=right:$c_2$] (c2) at (6.8,1.15) {};
\node[vertex,label=right:$c_3$] (c3) at (6.45,0) {};
\node at (5.45,-.12) {$\dots$};
\node[vertex,label=left:$c_t$] (ct) at (4.8,.15) {};
\draw (w)--(c2)--(c3)--(ct)--(w);
\draw (w)--(c3); \draw (c2)--(ct);
\node[draw,rounded corners,fit=(w)(c2)(c3)(ct),inner sep=7pt] (Kbox) {};
\node[below=3pt of Kbox] {$K_t$};
\draw[very thick,bend left=12] (a) to node[pos=.58,above=2pt] {$aw$} (w);
\end{tikzpicture}
\caption{The connected graph $F_t$.}
\label{fig:Ft}
\end{figure}

Because $K_t$ is a clique, each stable block of $F_t$ contains at most
one clique vertex.  Restricting to $H$ therefore gives a stable
partition $\mathcal B$ of $H$.  Conversely, each of
$c_2,\ldots,c_t$ may remain a singleton or enter a distinct block of
$\mathcal B$, while $w$ obeys the same rule except that it cannot enter
the block containing $a$.  Thus every stable partition of $F_t$ is
obtained uniquely from a stable partition of $H$ together with this
restricted partial injection.

The $29$ ordinary stable partitions of $H$ fall into the following
$12$ classes, where $B_*$ denotes the block containing $a$.
\begin{table}[htbp]
\centering
\small
\caption{Stable partitions of $H$ classified by type and $|B_*|$.}
\label{tab:Hclasses}
\begin{tabular}{c c r@{\qquad}c c r}
\toprule
Type & $|B_*|$ & Number & Type & $|B_*|$ & Number\\
\midrule
$(3,3)$&3&1 & $(3,2,1)$&3&3\\
$(3,2,1)$&2&2 & $(3,2,1)$&1&1\\
$(3,1,1,1)$&3&1 & $(3,1,1,1)$&1&1\\
$(2,2,2)$&2&1 & $(2,2,1,1)$&2&8\\
$(2,2,1,1)$&1&3 & $(2,1,1,1,1)$&2&3\\
$(2,1,1,1,1)$&1&4 & $(1^6)$&1&1\\
\bottomrule
\end{tabular}
\end{table}

For a chosen set of $r$ blocks, assigning distinct vertices from
$\{c_2,\ldots,c_t\}$ contributes the falling factorial
$(t-1)_r=(t-1)(t-2)\cdots(t-r)$.  Applying the preceding injection
rule to the $12$ seed classes and then multiplying by the semi-ordering
factor in \eqref{eq:semi-factor} gives the complete list below.  Write
\[
 A_i(t):=[m_{T_i}]X_{F_t}=(t-1)!q_i(t).
\]
For example, the unique seed partition of type $(3,3)$ gives
$A_1(t)=2(t-1)(t-1)!$.

\begin{table}[htbp]
\centering
\scriptsize
\setlength{\tabcolsep}{2.5pt}
\caption{All nonzero monomial types and normalized coefficients of $X_{F_t}$ for $t\ge6$.}
\label{tab:Ft23}
\begin{tabular}{c p{2.7cm} p{3.1cm} c p{2.7cm} p{3.2cm}}
\toprule
& $T_i$ & $q_i(t)$ & & $T_i$ & $q_i(t)$\\
\midrule
$T_1$ & $(4,4,1^{t-2})$ & $2(t-1)$
& $T_{13}$ & $(3,2,2,2,1^{t-3})$ & $6(28t^2-84t+57)$\\
$T_2$ & $(4,3,2,1^{t-3})$ & $6(t-1)$
& $T_{14}$ & $(3,2,2,1^{t-1})$ & $2(3t-2)(14t^2-26t+19)$\\
$T_3$ & $(4,3,1^{t-1})$ & $6t^2-9t+4$
& $T_{15}$ & $(3,2,1^{t+1})$ & $t(t+1)(28t^2-16t+13)$\\
$T_4$ & $(4,2,2,2,1^{t-4})$ & $12(t-1)$
& $T_{16}$ & $(3,1^{t+3})$ & $t(t+1)(t+2)(t+3)(7t-1)$\\
$T_5$ & $(4,2,2,1^{t-2})$ & $2(6t^2-11t+6)$
& $T_{17}$ & $(2,2,2,2,2,2,1^{t-6})$ & $720(t-1)$\\
$T_6$ & $(4,2,1^t)$ & $t(6t^2-4t+1)$
& $T_{18}$ & $(2,2,2,2,2,1^{t-4})$ & $240(t-1)(3t-8)$\\
$T_7$ & $(4,1^{t+2})$ & $t(t+1)(t+2)(2t-1)$
& $T_{19}$ & $(2,2,2,2,1^{t-2})$ & $24(15t^3-57t^2+83t-39)$\\
$T_8$ & $(3,3,3,1^{t-3})$ & $6(t-1)$
& $T_{20}$ & $(2,2,2,1^t)$ & $12t(10t^3-14t^2+19t-5)$\\
$T_9$ & $(3,3,2,2,1^{t-4})$ & $44(t-1)$
& $T_{21}$ & $(2,2,1^{t+2})$ & $2t(t+1)(t+2)(15t^2+8t+12)$\\
$T_{10}$ & $(3,3,2,1^{t-2})$ & $2(22t^2-54t+33)$
& $T_{22}$ & $(2,1^{t+4})$ & $6t(t+1)^2(t+2)(t+3)(t+4)$\\
$T_{11}$ & $(3,3,1^t)$ & $2t(11t^2-13t+7)$
& $T_{23}$ & $(1^{t+6})$ & $t(t+1)(t+2)(t+3)(t+4)(t+5)(t+6)$\\
$T_{12}$ & $(3,2,2,2,2,1^{t-5})$ & $168(t-1)$
& & & \\
\bottomrule
\end{tabular}
\end{table}

The support in \cref{tab:Ft23} is downward closed in dominance order.
Indeed, for $T_1=(4,4,1^{t-2})$ and any $\mu\vdash t+6$,
\[
 \mu\le T_1
 \quad\text{if and only if}\quad
 \mu_1\le4\ \text{ and }\ \ell(\mu)\ge t.
\]
The forward implication follows from the dominance inequalities for $k=1$ and $k=t-1$.  Conversely, if $\mu_1\le4$ and $\ell(\mu)\ge t$, then
for $2\le k\le t$ the first $k$ parts of $\mu$ sum to at most $k+6$,
which is the corresponding partial sum of $T_1$.  Listing these
possibilities gives exactly the $23$ rows of \cref{tab:Ft23}.

\subsection{The strongly nice property of \texorpdfstring{$F_t$}{Ft}}\label{subsec:Ft-strong}

A direct dominance comparison of the $23$ types gives exactly $33$
covering relations; they are listed together with their normalized
coefficient differences in \cref{tab:coverdiff}.  Every listed
difference $q_j(t)-q_i(t)$ for a cover $T_i\gtrdot T_j$ is
nonnegative for $t\ge6$.

\begin{theorem}\label{thm:Fstrong}
For every $t\ge6$, the connected graph $F_t$ is strongly nice.
\end{theorem}

\begin{proof}
Let $\lambda\ge\mu$ be partitions of $t+6$.  If
$[m_\lambda]X_{F_t}=0$, there is nothing to prove.  Otherwise
$\lambda$ and, by downward closure, $\mu$ occur in \cref{tab:Ft23}.
Any dominance relation between two occurring types is a chain of
covering relations, and \cref{tab:coverdiff} shows that the coefficient
weakly increases along every downward cover.  Hence
$[m_\lambda]X_{F_t}\le[m_\mu]X_{F_t}$.
\end{proof}

\subsection{A negative Schur coefficient}\label{subsec:Ft-nons}

We use the following bridge-transfer formula of Wang, Zhang, and
Zhao~\cite{WangZhangZhao2026}.  If $\ell(\lambda)\le t$, pad $\lambda$
with zeros to length $t$ and write
$\lambda+(1^t)=(\lambda_1+1,\ldots,\lambda_t+1)$.

\begin{theorem}[Wang--Zhang--Zhao~\cite{WangZhangZhao2026}]\label{thm:bridge-transfer}
Let $B_t(G,v)$ be obtained from $G\sqcup K_t$ by choosing
$w\in V(K_t)$ and adding the single edge $vw$.  If
$\lambda\vdash |V(G)|$ and $\ell(\lambda)\le t$, then
\[
 [s_{\lambda+(1^t)}]X_{B_t(G,v)}
 =(t-1)(t-1)!\,[s_\lambda]X_G.
\]
\end{theorem}

\begin{theorem}\label{thm:FnonS}
For every $t\ge6$,
\[
 [s_{(3,3,3,1^{t-3})}]X_{F_t}
 =-4(t-1)(t-1)!<0.
\]
Consequently, $F_t$ is connected and strongly nice but not Schur positive.
\end{theorem}

\begin{proof}
Since $F_t=B_t(H,a)$ and
$(2,2,2)+(1^t)=(3,3,3,1^{t-3})$, \cref{thm:bridge-transfer} and
$[s_{222}]X_H=-4$ give the displayed coefficient.  The strongly nice
property follows from \cref{thm:Fstrong}.
\end{proof}

\section{Disconnected strongly nice but not Schur positive graphs}\label{sec:disconnected-strong-not-s}

We next study disjoint unions.  The key point is that the strongly nice
property is closed under multiplication for the class of symmetric
functions relevant here.

\subsection{A multiplicative closure property}

\begin{theorem}\label{thm:product}
Let $f$ and $g$ be homogeneous symmetric functions with nonnegative
monomial coefficients.  If both are strongly nice, then $fg$ is
strongly nice.  Consequently, the disjoint union of two strongly nice
graphs is strongly nice.
\end{theorem}

\begin{proof}
For a weak composition $\alpha$ in sufficiently many variables, write
$F(\alpha)=[x^\alpha]f$ and $G(\alpha)=[x^\alpha]g$.  Symmetry and the
strongly nice property imply that these coefficients do not decrease
under a unit balancing transfer.  If $h=fg$ and
$H(\gamma)=[x^\gamma]h$, then
\[
 H(\gamma)=\sum_{\alpha+\beta=\gamma}F(\alpha)G(\beta).
\]
Fix all splittings outside two coordinates with
$\gamma_i=p$, $\gamma_j=q$, and $p\ge q+2$.  Homogeneity fixes
$s=\alpha_i+\alpha_j$.  As $r=\alpha_i$ varies, the corresponding
values of $F$ form a nonnegative symmetric unimodal sequence centered
at $s/2$; the analogous values of $G$ form such a sequence centered at
$(p+q-s)/2$.  Their convolution is again symmetric unimodal (for
example, decompose each sequence into nonnegative combinations of
centered interval indicators).  Hence the contribution cannot decrease
when $(p,q)$ is replaced by $(p-1,q+1)$.  Summing over the fixed outside
splittings gives the same inequality for $H$.  Since every dominance
relation is generated by unit balancing transfers, $fg$ is strongly
nice.

For graphs, use $X_{G_1\sqcup G_2}=X_{G_1}X_{G_2}$ and the
nonnegativity of the monomial coefficients from
\cref{prop:stanley-m}.
\end{proof}

\subsection{Disjoint union with a complete graph}

Write $X_G=\sum_\lambda c_\lambda(G)s_\lambda$.  The
Littlewood--Richardson rule shows that the disjoint union of two Schur
positive graphs is Schur positive.  If one factor has a negative Schur
coefficient, however, the product may still be Schur positive because
different contributions can cancel.

\begin{proposition}\label{prop:Kt-pieri}
For every graph $G$, every $t\ge1$, and every partition $\nu$,
\[
 [s_\nu]X_{G\sqcup K_t}
 =t!\sum_{\substack{\lambda\\
 \nu/\lambda\text{ is a vertical }t\text{-strip}}}c_\lambda(G).
\]
Thus $G\sqcup K_t$ is Schur positive exactly when all these sums are
nonnegative.  In particular, if a negative coefficient
$c_{\lambda_0}(G)$ is the unique nonzero term contributing to one such
sum, then $G\sqcup K_t$ is not Schur positive.
\end{proposition}

\begin{proof}
Since $X_{K_t}=t!e_t$, the formula is the vertical Pieri rule applied
to $t!X_Ge_t$.
\end{proof}

Recall $H=K_{3,3}-e$ and define
\[
 M_t:=H\sqcup K_t,\qquad t\ge1.
\]

\begin{figure}[htbp]
\centering
\begin{tikzpicture}[scale=0.95]
% H = K_{3,3}-e
\node at (1.2,2.1) {$H=K_{3,3}-e$};
\foreach \name/\x/\y in {a/0/1.2,b/1.2/1.2,c/2.4/1.2,d/0/0,e/1.2/0,f/2.4/0}
  \node[vertex,label={[font=\scriptsize]90:\name}] (\name) at (\x,\y) {};
\foreach \u/\v in {a/e,a/f,b/d,b/e,b/f,c/d,c/e,c/f}
  \draw[edge] (\u)--(\v);
\node at (4.0,0.6) {$\sqcup$};
% K_t schematic
\node at (7.2,2.1) {$K_t$};
\foreach \name/\x/\y in {u1/5.8/1.05,u2/6.8/1.65,u3/7.9/1.2,u4/7.2/0.15}
  \node[vertex] (\name) at (\x,\y) {};
\foreach \u/\v in {u1/u2,u1/u3,u1/u4,u2/u3,u2/u4,u3/u4}
  \draw[edge] (\u)--(\v);
\node at (8.55,1.0) {$\cdots$};
\end{tikzpicture}
\caption{The disconnected family $M_t=H\sqcup K_t$.}
\label{fig:Mt-family}
\end{figure}
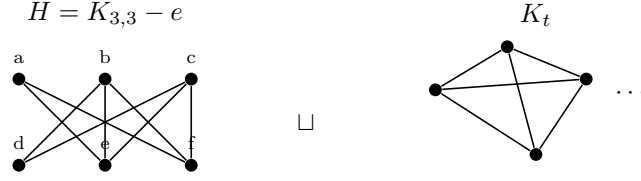

\begin{theorem}\label{thm:Mt}
Every $M_t$ is disconnected and strongly nice.  Moreover, $M_t$ is
Schur positive if and only if $t=1$ or $t=2$.  See Figure~\ref{fig:Mt-family}.  For every $t\ge3$,
\[
 [s_{(3,3,3,1^{t-3})}]X_{M_t}=-4t!<0.
\]
\end{theorem}

\begin{proof}
The graph $K_t$ is Schur positive and hence strongly nice, while $H$ is strongly nice by \cref{subsec:Hseed}.  Therefore \cref{thm:product} shows that every $M_t$ is strongly nice.  For $t=1,2$, applying Pieri to \eqref{eq:H-s} gives
\[
\begin{aligned}
X_{M_1}={}&2s_{43}+4s_{421}+2s_{41^3}+6s_{331}+32s_{321^2}
 +54s_{31^4}+22s_{2^31}\\
&+78s_{2^21^3}+204s_{21^5}+152s_{1^7},
\end{aligned}
\]
and
\[
\begin{aligned}
\tfrac12X_{M_2}={}&2s_{44}+6s_{431}+4s_{422}+6s_{421^2}+2s_{41^4}
 +32s_{331^2}+28s_{32^21}+84s_{321^3}\\
&+54s_{31^5}+26s_{2^4}+74s_{2^31^2}+230s_{2^21^4}
 +204s_{21^6}+152s_{1^8},
\end{aligned}
\]
so both are Schur positive.

For $t\ge3$, put $\nu=(3,3,3,1^{t-3})$.  This partition has exactly
$t$ rows.  If $\nu/\lambda$ is a vertical $t$-strip with
$\lambda\vdash6$, one cell must be removed from every row, so
$\lambda=(2,2,2)$.  By \cref{prop:Kt-pieri} and
$[s_{222}]X_H=-4$,
\[
 [s_\nu]X_{M_t}=t![s_{222}]X_H=-4t!.
\]
\end{proof}

Thus a non-Schur-positive component does not by itself determine the
Schur positivity of a disjoint union: $H$ is not Schur positive,
whereas $H\sqcup K_1$ and $H\sqcup K_2$ are Schur positive; for
$t\ge3$, the graphs $M_t$ give the desired disconnected separating
family.

\section{Connected nice but not strongly nice graphs}\label{sec:nice-not-strong}

\subsection{The family \texorpdfstring{$N_r$}{Nr}}

For vertex-disjoint graphs $G$ and $H$, let $G\vee H$ denote their
join, obtained from $G\sqcup H$ by adding every edge between the two
vertex sets.  For $r\ge2$, define
\begin{equation}\label{eq:Nr-def}
 N_r:=K_r\vee(K_2\sqcup2K_1).
\end{equation}

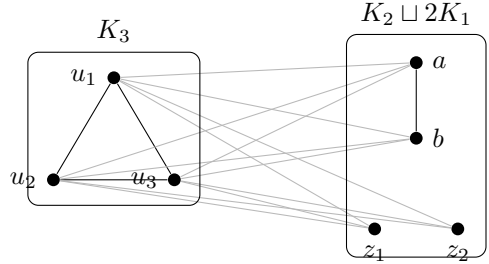
\begin{figure}[htbp]
\centering
\begin{tikzpicture}[scale=1]
\node[vertex,label=left:$u_1$] (u1) at (0,1.35) {};
\node[vertex,label=left:$u_2$] (u2) at (-.8,0) {};
\node[vertex,label=left:$u_3$] (u3) at (.8,0) {};
\draw (u1)--(u2)--(u3)--(u1);
\node[vertex,label=right:$a$] (a) at (4,1.55) {};
\node[vertex,label=right:$b$] (b) at (4,.55) {};
\node[vertex,label=below:$z_1$] (z1) at (3.45,-.65) {};
\node[vertex,label=below:$z_2$] (z2) at (4.55,-.65) {};
\draw (a)--(b);
\foreach \u in {u1,u2,u3}{\foreach \v in {a,b,z1,z2}{\draw[gray!55,thin] (\u)--(\v);}}
\node[draw,rounded corners,fit=(u1)(u2)(u3),inner sep=7pt,label=above:$K_3$] {};
\node[draw,rounded corners,fit=(a)(b)(z1)(z2),inner sep=8pt,label=above:$K_2\sqcup2K_1$] {};
\end{tikzpicture}
\caption{The graph $N_3$, illustrating the construction
$N_r=K_r\vee(K_2\sqcup2K_1)$.}
\label{fig:Nr-family}
\end{figure}

Figure~\ref{fig:Nr-family} illustrates the construction for $r=3$. Every vertex of $K_r$ is universal in $N_r$ and hence is a singleton
block in every stable partition.  If $Q=K_2\sqcup2K_1$, its ordinary
stable partitions have types $(3,1),(2,2),(2,1,1),(1^4)$ with counts
$2,2,5,1$, respectively.  Using \eqref{eq:semi-factor}, we obtain
\begin{equation}\label{eq:Nr-m}
\begin{aligned}
X_{N_r}={}&2(r+1)!\,m_{(3,1^{r+1})}
+4r!\,m_{(2,2,1^r)}\\
&+5(r+2)!\,m_{(2,1^{r+2})}
+(r+4)!\,m_{(1^{r+4})}.
\end{aligned}
\end{equation}

\begin{theorem}\label{thm:Nr}
For every $r\ge2$, the graph $N_r$ is connected and nice but not
strongly nice.
\end{theorem}

\begin{proof}
The four types in \eqref{eq:Nr-m} are exactly the partitions below
$(3,1^{r+1})$ and form the dominance chain
\[
 (3,1^{r+1})>(2,2,1^r)>(2,1^{r+2})>(1^{r+4}),
\]
so $N_r$ is nice.  On the first cover,
\[
 2(r+1)!-4r!=2(r-1)r!>0
\]
for $r\ge2$, which reverses the inequality required by the strongly
nice property.  Connectedness is immediate from the join construction.
\end{proof}

\subsection{The level-\texorpdfstring{$k$}{k} nice property}

The nice property records only whether a monomial coefficient is
positive, while the strongly nice property compares its full value.
This suggests the following finite-level refinement.

\begin{definition}\label{def:levelk}
Let $k\ge1$. A graph $G$ is \emph{level-$k$ nice} if, for all
$\lambda,\mu\vdash |V(G)|$ with $\lambda\ge\mu$ and every integer
$1\le j\le k$, the inequality $a_\lambda(G)\ge j$ implies
$a_\mu(G)\ge j$.
\end{definition}

Equivalently,
\begin{equation}\label{eq:level-trunc}
 \min\{a_\lambda(G),k\}\le\min\{a_\mu(G),k\}
 \qquad(\lambda\ge\mu).
\end{equation}

\begin{proposition}\label{prop:level-hierarchy}
The level-$1$ nice property is the nice property, and every level-$(k+1)$ nice
graph is level-$k$ nice.  A graph is strongly nice if and only if it
is level-$k$ nice for every $k\ge1$.
\end{proposition}

\begin{proof}
The first two statements follow directly from the definition.  A
strongly nice graph satisfies \eqref{eq:level-trunc} for every $k$.
Conversely, if $G$ is level-$k$ nice for every $k$ and
$\lambda\ge\mu$, take $k=a_\lambda(G)$ when this coefficient is
positive; then the level-$k$ condition gives
$a_\mu(G)\ge a_\lambda(G)$.  The case $a_\lambda(G)=0$ is automatic.
\end{proof}

\begin{definition}\label{def:depth}
The \emph{level depth} of $G$ is
\[
 d(G):=\sup\bigl(\{k\in\N:G\text{ is level-}k\text{ nice}\}\cup\{0\}\bigr),
\]
with $d(G)=\infty$ when the set is unbounded.  Thus a graph that is not
nice has depth $0$, and a graph is strongly nice exactly when its depth
is infinite.
\end{definition}

\begin{theorem}\label{thm:Nr-depth}
For every $r\ge2$,
\[
 d(N_r)=4r!.
\]
Hence connected nice graphs that are not strongly nice can have
arbitrarily large finite level depth.
\end{theorem}

\begin{proof}
In \eqref{eq:Nr-m}, the only reversed comparison is
$2(r+1)!>4r!$ on the first cover; all remaining downward comparisons
are in the required direction.  For $k\le4r!$, truncation makes the
first two coefficients equal whenever needed, so $N_r$ is level-$k$
nice.  At $k=4r!+1$ the first cover gives
\[
 \min\{2(r+1)!,k\}=4r!+1>4r!=\min\{4r!,k\},
\]
so $d(N_r)=4r!$.
\end{proof}

The first finite level is already strict.  For
$A=K_{1,4}\sqcup2K_1$, every partition of $7$ except $(7)$ occurs as a
stable-partition type, so $A$ is nice, but
$(5,2)>(4,3)$ and
$a_{(5,2)}(A)=2>a_{(4,3)}(A)=1$; hence $A$ is not level-$2$ nice.
On the other hand, $N_2$ is level-$8$ nice but not strongly nice.
Thus the level-$2$ nice property is strictly stronger than the nice property
and strictly weaker than the strongly nice property.

\appendix
\section{The 33 covering differences for \texorpdfstring{$F_t$}{Ft}}\label{app:cover-differences}

For $t\ge6$, write $A_i(t)=(t-1)!q_i(t)$ as in
\cref{tab:Ft23}.  A direct dominance check gives the following $33$
covers.  For each $T_i\gtrdot T_j$, the table records
$\Delta_{ij}(t)=q_j(t)-q_i(t)$.

{\scriptsize
\setlength{\tabcolsep}{4pt}
\renewcommand{\arraystretch}{.92}
\begin{longtable}{c c l}
\caption{Normalized coefficient differences on the covering relations.}\label{tab:coverdiff}\\
\toprule
Upper & Lower & $\Delta_{ij}(t)$\\
\midrule
\endfirsthead
\toprule
Upper & Lower & $\Delta_{ij}(t)$\\
\midrule
\endhead
$T_1$&$T_2$&$4(t-1)$\\
$T_2$&$T_3$&$6t^2-15t+10$\\
$T_2$&$T_4$&$6(t-1)$\\
$T_2$&$T_8$&$0$\\
$T_3$&$T_5$&$6t^2-13t+8$\\
$T_4$&$T_5$&$2(2t-3)(3t-4)$\\
$T_4$&$T_9$&$32(t-1)$\\
$T_5$&$T_6$&$6t^3-16t^2+23t-12$\\
$T_5$&$T_{10}$&$2(t-1)(16t-27)$\\
$T_6$&$T_7$&$t(2t^3-t^2+5t-3)$\\
$T_6$&$T_{11}$&$t(16t^2-22t+13)$\\
$T_7$&$T_{15}$&$t(t+1)(26t^2-19t+15)$\\
$T_8$&$T_9$&$38(t-1)$\\
$T_9$&$T_{10}$&$2(22t^2-76t+55)$\\
$T_9$&$T_{12}$&$124(t-1)$\\
$T_{10}$&$T_{11}$&$2(11t^3-35t^2+61t-33)$\\
$T_{10}$&$T_{13}$&$4(31t^2-99t+69)$\\
$T_{11}$&$T_{14}$&$2(31t^3-93t^2+102t-38)$\\
$T_{12}$&$T_{13}$&$6(28t^2-112t+85)$\\
$T_{12}$&$T_{17}$&$552(t-1)$\\
$T_{13}$&$T_{14}$&$2(42t^3-190t^2+361t-209)$\\
$T_{13}$&$T_{18}$&$6(92t^2-356t+263)$\\
$T_{14}$&$T_{15}$&$28t^4-72t^3+209t^2-205t+76$\\
$T_{14}$&$T_{19}$&$2(138t^3-578t^2+887t-430)$\\
$T_{15}$&$T_{16}$&$t(t+1)(7t^3+6t^2+53t-19)$\\
$T_{15}$&$T_{20}$&$t(92t^3-180t^2+231t-73)$\\
$T_{16}$&$T_{21}$&$t(t+1)(t+2)(23t^2-4t+27)$\\
$T_{17}$&$T_{18}$&$240(t-1)(3t-11)$\\
$T_{18}$&$T_{19}$&$24(15t^3-87t^2+193t-119)$\\
$T_{19}$&$T_{20}$&$12(10t^4-44t^3+133t^2-171t+78)$\\
$T_{20}$&$T_{21}$&$2t(15t^4-7t^3+150t^2-62t+54)$\\
$T_{21}$&$T_{22}$&$2t(t+1)(t+2)(3t^3+9t^2+49t+24)$\\
$T_{22}$&$T_{23}$&$t(t+1)(t+2)(t+3)(t+4)(t^2+5t+24)$\\
\bottomrule
\end{longtable}
}

Every displayed difference is nonnegative for $t\ge6$.  For the
polynomial factors whose positivity is not immediate, substituting
$t=u+6$ expands them as polynomials in $u\ge0$ with nonnegative
coefficients.  Hence all $33$ covering inequalities used in
\cref{thm:Fstrong} hold for $t\ge6$.

\end{document}